\documentclass[preprint,12pt]{elsarticle}

 \usepackage{amssymb}
 
 \usepackage{amsmath}
 \usepackage{enumitem}
 
 \usepackage{amsthm}

 \usepackage{tikz} 
 \usepackage{tikz-cd}
  \usetikzlibrary{arrows.meta,bending,matrix} 
 
 \usepackage{subcaption}

 \usepackage{todonotes}
 \usepackage{soul}
 \usepackage{xspace}
 \usepackage{xcolor}

\theoremstyle{definition}
\newtheorem{defn}{Definition}
\usepackage{tikz-cd}
\newtheorem{thm}{Theorem}
\theoremstyle{definition} % Cambia el estilo de los corolarios
\usepackage{comment}

\newtheorem{lem}{Lemma}

\newcommand{\stopmin}{\sigma_{R}}
\newcommand{\stopmax}{\sigma_{L}}
\newcommand{\stopmisub}{{R}}
\newcommand{\stopmasub}{{L}}
\begin{document}
	\begin{frontmatter}

		\title{A characterization of interval nest digraphs}
		
		%% use optional labels to link authors explicitly to addresses:
		%% \author[label1,label2]{}
		%% \affiliation[label1]{organization={},
			%%             addressline={},
			%%             city={},
			%%             postcode={},
			%%             state={},
			%%             country={}}
		%%
		%% \affiliation[label2]{organization={},
			%%             addressline={},
			%%             city={},
			%%             postcode={},
			%%             state={},
			%%             country={}}
		
		\author[a,e]{Ayelén Alcantar } %% Author name
		\author[b,d]{Flavia Bonomo }
		\author[a,e,f]{Guillermo Durán }
		\author[g]{Nina Pardal}
		%% Author affiliation
		
		\affiliation[a]{organization={Universidad de Buenos Aires, Facultad de Ciencias Exactas y Naturales, Departamento de Matemática},
			city={Buenos Aires},
			country={Argentina}}
		\affiliation[b]{organization={Universidad de Buenos Aires, Facultad de Ciencias Exactas y Naturales, Departamento de Computación},
			city={Buenos Aires},
			country={Argentina}}
		%\affiliation[c]{organization={CONICET},%Department and Organization
			%country={Argentina}}
		\affiliation[e]{organization={CONICET-Universidad de Buenos Aires, Instituto de Cálculo (IC)},%Department and Organization
			%addressline={}, 
			city={Buenos Aires},
			%postcode={}, 
			%state={},
			country={Argentina}}
            \affiliation[d]{organization={CONICET-Universidad de Buenos Aires,  Instituto de
Investigación en Ciencias de la Computación (ICC)},%Department and Organization
			%addressline={}, 
			city={Buenos Aires},
			%postcode={}, 
			%state={},
			country={Argentina}}
		\affiliation[f]{organization={Universidad de Chile, Facultad de Ciencias Físicas y Matemáticas, Departamento de Ingeniería Industrial},%Department and Organization
			%addressline={}, 
			city={Santiago},
			%postcode={}, 
			%state={},
			country={Chile}
		}
		
		\affiliation[g]{organization={ School of Informatics, University of Edinburgh},%Department and Organization
			%addressline={}, 
			city={Edinburgh},
			%postcode={}, 
			%state={},
			country={United Kingdom}
		}
		%% Abstract
		\begin{abstract}

			A digraph consisting of a set of vertices $V$ and a set of arcs $E$ is called an \emph{interval digraph} if there exists a family of closed intervals $\{I_u,J_u\}_{u \in V}$ such that $uv$ is an arc if and only 
			if the intersection of $I_u$ and $J_v$ is non-empty.
			Interval digraphs naturally generalize interval graphs, by extending the classical interval intersection model to directed graphs.
			Several subclasses of interval digraphs have been studied in the literature---such as balanced, chronological and catch interval digraphs---each characterized by admitting interval representations that satisfy specific restrictions. Among these, \emph{interval nest digraphs} are the ones that admit an interval representation in which $J_u$ is contained in $I_u$ for all vertices $u$ of $V$. 
			
			In this work, we provide a complete characterization of interval nest digraphs in terms of vertex linear orderings with forbidden patterns, which we call \emph{nest orderings}. 
            This result completes the picture of vertex-ordering characterizations among the main subclasses of interval digraphs.
		\end{abstract}

		
	\end{frontmatter}
	
	\section{Introduction}
	Interval graphs were introduced by Haj\'os in 1957~\cite{Haj-int} as the intersection graphs of intervals of the real line, i.e., every vertex represents one interval and two vertices are adjacent if and only if the corresponding intervals intersect. Since their introduction, they
    are one of the most studied graph classes~\cite{BoothLueker1976,B-L-S-classes,Fishburn-book,Golumbic1980,golumbic-interval,Kobler2011, Kiyomi2010, FulkersonGross1965}. 
	They arise naturally in the modeling of various real-world applications, particularly those involving time dependencies or other constraints that are linear in nature~\cite{Benzer1959,Cohen1979Interval,Huber2007Graphs}.
	Lekkerkerker and Boland~\cite{Lekkerkerker62} characterized interval graphs by forbidden induced subgraphs, showing that they are precisely those chordal graphs that contain no asteroidal triple.
	Interval digraphs extend the notion of interval graphs to the directed setting, allowing the representation of asymmetric relationships. Introduced by Das, Sen, Roy, and West in 1989~\cite{S-D-R-W}, the class of interval digraphs is of interest not only from a theoretical viewpoint in structural graph theory but also from a practical one, with applications to neural networks, biological systems, and communication networks~\cite{Ceyhan2011,dynamics2021}.
	For these reasons, various subclasses of interval digraphs have been proposed, each imposing additional structural restrictions to the intersection models that yield elegant characterizations and, in many cases, efficient algorithms.

	Formally, a digraph $D$ is an ordered pair $D = (V, E)$, where $V$ is the set of vertices and $E$ is the set of arcs.  
	An arc $uv$ denotes a directed edge from vertex $u$ to vertex $v$. We denote by $\mathcal{S}(E)$ the set of all symmetric arcs, that is, the set of arcs $uv \in E$ such that $vu \in E$.
	A vertex $u$ has a \textit{loop} if $uu\in E$. The \emph{out-neighborhood} of a vertex $u$ is defined as $N^+(u) = \{v \in V : uv \in E\}$. 
    
	A digraph $D = (V, E)$ is an \emph{interval digraph} if there exists a family of closed intervals $\{I_u, J_u\}_{u \in V}$ over the real line such that, for every $u, v \in V$, not necessarily distinct, $uv \in E$ if and only if $I_u \cap J_v \neq \emptyset$~\cite{S-D-R-W}. The intervals $I_u$ and $J_u$ are referred to as the \emph{origin} and \emph{destination} intervals, respectively, and the family $\{I_u, J_u\}_{u \in V}$ is called the \emph{(interval) (intersection) model} or \emph{(interval) representation}. Several restrictions imposed to the interval models give rise to well-studied subclasses of interval digraphs. The \textit{reflexive interval digraphs} form a large subclass with the constraint that every vertex has a loop, i.e., $J_u \cap I_u \neq \emptyset$ for all $u \in V$. 
    The \emph{interval point digraphs} are those in which $J_u$ is a point for all $u \in V$ (Das, Sen, Roy, and West~\cite{S-D-R-W}). 
	When the destination interval $J_u$ is an arbitrary point \emph{inside} $I_u$ for every $u \in V$, one obtains an \emph{interval catch digraph}, and if $J_u$ is defined as the midpoint of $I_u$ for every $u \in V$, the digraph is a \emph{balanced interval digraph} (also referred to as \emph{central interval catch digraphs}) (Maehara~\cite{Mae-dig}).  If, instead, $J_u$ is defined to be the left endpoint of $I_u$ for every $u \in V$, the resulting structure is a \emph{chronological interval digraph} (Das, Francis, Hell, and Huang~\cite{SandipchronID}). 
	When $J_u$ is not necessarily a point but the intervals must satisfy that $I_u$ and $J_u$ have the same left endpoint for each vertex $u \in V$, the digraph is known as an \emph{adjusted interval digraph} (Feder, Hell, Huang, and Rafiey~\cite{Feder2009}).
	
	The present work focuses on the class of \emph{interval nest digraphs}, introduced by Prisner in~\cite{Prisner1994}. In this subclass, the representation $\{I_u, J_u\}_{u \in V}$ satisfies that each destination interval is contained in its corresponding origin interval, i.e., $J_u \subseteq I_u$ for all $u \in V$. In particular, interval nest digraphs are reflexive.
	
    Several of these subclasses admit elegant characterizations based on structural properties of vertex orderings. Most of them can be represented as a family of forbidden patterns. Forbidden patterns are usually depicted using solid and dashed arcs: for vertices $u,v$, a solid arc depicted from $u$ to $v$ indicates the presence of the arc $(u,v)$ in the digraph; a dashed arc from $u$ to $v$ indicates the absence of the arc $(u,v)$ in the digraph; the absence of a depicted arc from $u$ to $v$ indicates that the arc $(u,v)$ can be either present or not in the digraph.
    For instance, in~\cite{Feder2009}, Feder, Hell, Huang, and Rafiey characterize adjusted interval digraphs by means of a min-ordering: for any three vertices $u<v<w$, this ordering ensures that if $uw\in E$, then $uv\in E$, and if $wu\in E$, then $vu\in E$. Similarly, in~\cite{Mae-dig}, Maehara characterizes interval catch digraphs using a total order in which, for $u<v<w$, the presence of the arc $uw$ implies $uv\in E$, and the presence of $wu$ implies $wv\in E$. 
    See Figure~\ref{fig:adj-catch} for a graphical representation of the forbidden patterns. 

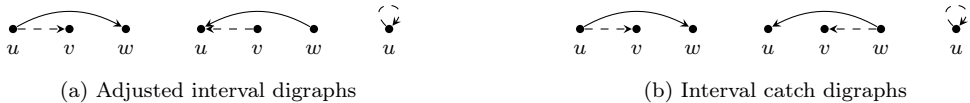
\begin{figure}[h!]
	\centering
	
	\begin{subfigure}{0.45\textwidth}
		\centering
		\begin{tikzpicture}[scale=0.5, ->, >=stealth, font=\scriptsize]
			\node (v0) at (-1.5,0) [circle, fill=black, inner sep=1.1pt, label=below:$u$] {};
			\node (v1) at (0,0)  [circle, fill=black, inner sep=1.1pt, label=below:$v$] {};
			\node (v2) at (1.5,0)  [circle, fill=black, inner sep=1.1pt, label=below:$w$] {};
			\draw (v0) to[bend left=30] (v2);
			\draw [dashed, ->] (v0) to (v1);
        
			\node (w0) at (3.5,0) [circle, fill=black, inner sep=1.1pt, label=below:$u$] {};
			\node (w1) at (5,0)  [circle, fill=black, inner sep=1.1pt, label=below:$v$] {};
			\node (w2) at (6.5,0)  [circle, fill=black, inner sep=1.1pt, label=below:$w$] {};
			\draw (w2) to[bend right=30] (w0);
			\draw [dashed, ->] (w1) to (w0);
        
            \node (z0) at (8.5,0) [circle, fill=black, inner sep=1.1pt, label=below:$u$] {};
            \draw [->, dashed] (z0) edge [out=135,in=45,looseness=15] (z0);

		\end{tikzpicture}
		\caption{Adjusted interval digraphs}
		\label{fig:adjusted}
	\end{subfigure}
	\hfill
	\begin{subfigure}{0.45\textwidth}
		\centering
		\begin{tikzpicture}[scale=0.5, ->, >=stealth, font=\scriptsize]
			\node (v0) at (-1.5,0) [circle, fill=black, inner sep=1.1pt, label=below:$u$] {};
			\node (v1) at (0,0)  [circle, fill=black, inner sep=1.1pt, label=below:$v$] {};
			\node (v2) at (1.5,0)  [circle, fill=black, inner sep=1.1pt, label=below:$w$] {};
			\draw (v0) to[bend left=30] (v2);
			\draw [dashed, ->] (v0) to (v1);
        
			\node (w0) at (3.5,0) [circle, fill=black, inner sep=1.1pt, label=below:$u$] {};
			\node (w1) at (5,0)  [circle, fill=black, inner sep=1.1pt, label=below:$v$] {};
			\node (w2) at (6.5,0)  [circle, fill=black, inner sep=1.1pt, label=below:$w$] {};
			\draw (w2) to[bend right=30] (w0);
			\draw [dashed, ->] (w2) to (w1);
        
            \node (z0) at (8.5,0) [circle, fill=black, inner sep=1.1pt, label=below:$u$] {};
            \draw [->, dashed] (z0) edge [out=135,in=45,looseness=15] (z0);

		\end{tikzpicture}
		\caption{Interval catch digraphs}
		\label{fig:catch}
	\end{subfigure}
	
	\caption{Forbidden patterns for adjusted interval digraphs (left) and interval catch digraphs (right), where $u < v < w$, extracted from the characterizations in~\cite{Feder2009} and~\cite{Mae-dig}, respectively.}
	\label{fig:adj-catch}
\end{figure}

The forbidden patterns for interval point digraphs (Figure~\ref{fig:point}) arise directly from their definition, as observed by Das, Sen, Roy, and West in~\cite{S-D-R-W}, by noticing that ordering the vertices according to their destination point (and breaking ties arbitrarily), the out-neighbors of any vertex have to be consecutive in the order, and that property ensures also the existence of a suitable model. 

\begin{figure}[h]
\centering
\begin{tikzpicture}[font=\scriptsize]
\matrix[row sep=3mm, column sep=5mm]{

\node{\begin{tikzpicture}[scale=0.5, ->, >=stealth]
\node (v0) at (-1.5,0) [circle, fill=black, inner sep=1.1pt, label=below:$u$] {};
\node (v1) at (0,0)  [circle, fill=black, inner sep=1.1pt, label=below:$v$] {};
\node (v2) at (1.5,0)  [circle, fill=black, inner sep=1.1pt, label=below:$w$] {};
\node (v3) at (3,0)  [circle, fill=black, inner sep=1.1pt, label=below:$z$] {};
\draw (v0) to[bend left=38] (v3);
\draw [dashed, ->] (v0) to[bend left=30] (v2);
\draw (v0) to (v1);
\node at (0.6,-1.4) {(p)};
\end{tikzpicture}
}; &

\node{
\begin{tikzpicture}[scale=0.5, ->, >=stealth]
\node (v0) at (-1.5,0) [circle, fill=black, inner sep=1.1pt, label=below:$u$] {};
\node (v1) at (0,0)  [circle, fill=black, inner sep=1.1pt, label=below:$v$] {};
\node (v2) at (1.5,0)  [circle, fill=black, inner sep=1.1pt, label=below:$w$] {};
\node (v3) at (3,0)  [circle, fill=black, inner sep=1.1pt, label=below:$z$] {};
\phantom{\draw (v0) to[bend left=38] (v3);}
\draw (v1) to[bend left=38] (v3);
\draw (v1) to  (v0);
\draw [dashed, ->] (v1) to (v2);
\node at (0.6,-1.4) {(q)};
\end{tikzpicture}
}; & 

\node{
\begin{tikzpicture}[scale=0.5, ->, >=stealth]
\node (v0) at (-1.5,0) [circle, fill=black, inner sep=1.1pt, label=below:$u$] {};
\node (v1) at (0,0)  [circle, fill=black, inner sep=1.1pt, label=below:$v$] {};
\node (v2) at (1.5,0)  [circle, fill=black, inner sep=1.1pt, label=below:$w$] {};
\node (v3) at (3,0)  [circle, fill=black, inner sep=1.1pt, label=below:$z$] {};
\phantom{\draw (v0) to[bend left=38] (v3);}
\draw (v2) to[bend right=38] (v0);
\draw [dashed, ->] (v2) to  (v1);
\draw  (v2) to  (v3);
\node at (0.6,-1.4) {(r)};
\end{tikzpicture}
}; & 

\node{
\begin{tikzpicture}[scale=0.5, ->, >=stealth]
\node (v0) at (-1.5,0) [circle, fill=black, inner sep=1.1pt, label=below:$u$] {};
\node (v1) at (0,0)  [circle, fill=black, inner sep=1.1pt, label=below:$v$] {};
\node (v2) at (1.5,0)  [circle, fill=black, inner sep=1.1pt, label=below:$w$] {};
\node (v3) at (3,0)  [circle, fill=black, inner sep=1.1pt, label=below:$z$] {};
\draw (v3) to[bend right=38] (v0);
\draw [dashed, ->] (v3) to[bend right=30] (v1);
\draw (v3) to (v2);
\node at (0.6,-1.4) {(s)};
\end{tikzpicture}
}; \\
};
\end{tikzpicture}

\caption{Forbidden patterns characterizing interval point digraphs, where possibly $u= v$ in~(p), $v = w$ in~(q) and~(r), and $w=z$ in~(s)~\cite{S-D-R-W}.}
\label{fig:point}
\end{figure}
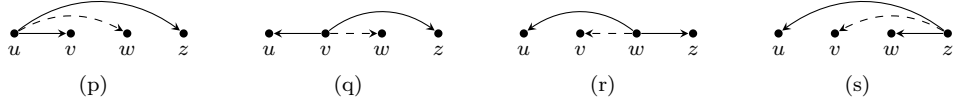

    In~\cite{sanchita-ghos}, Paul and Ghosh characterize the class of balanced interval digraphs as those admitting an injective labeling $f: V \to \mathbb{R}_{\geq 0}$ such that $|f(v_i)-f(v_j)| < |f(v_i)-f(v_k)|$ for all vertices $v_i,v_j,v_k$, whenever $v_iv_j \in E$ and $v_iv_k \not \in E$. They also show that every balanced interval digraph admits an ordering $v_1, \ldots, v_n$ that, in addition to satisfying the general interval catch digraph conditions defined above, ensures that for any $i<j$, either $i_1 \le j_1$ or $i_2 \le j_2$, where $i_1$ and $i_2$ denote, respectively, the smallest and largest indices such that $i_1=i$ or $v_i v_{i_1}\in E$, and $i_2=i$ or $v_i v_{i_2}\in E$, for each $i=1,2,\ldots,n$. See Figure~\ref{fig:balanced} for an equivalent representation of this necessary condition as forbidden patterns. 

\begin{figure}[h!]
		\centering
		\begin{tikzpicture}[scale=0.5, ->, >=stealth, font=\scriptsize]
			\node (v0) at (-2,0) [circle, fill=black, inner sep=1.1pt, label=below:$u$] {};
			\node (v1) at (0,0)  [circle, fill=black, inner sep=1.1pt, label=below:$v$] {};
			\node (v2) at (2,0)  [circle, fill=black, inner sep=1.1pt, label=below:$w$] {};
			\draw (v0) to[bend left=30] (v2);
			\draw [dashed, ->] (v0) to (v1);
        
			\node (w0) at (5,0) [circle, fill=black, inner sep=1.1pt, label=below:$u$] {};
			\node (w1) at (7,0)  [circle, fill=black, inner sep=1.1pt, label=below:$v$] {};
			\node (w2) at (9,0)  [circle, fill=black, inner sep=1.1pt, label=below:$w$] {};
			\draw (w2) to[bend right=30] (w0);
			\draw [dashed, ->] (w2) to (w1);

			\node (x0) at (12,0) [circle, fill=black, inner sep=1.1pt, label=below:$u$] {};
			\node (x1) at (14,0)  [circle, fill=black, inner sep=1.1pt, label=below:$v$] {};
			\node (x2) at (16,0)  [circle, fill=black, inner sep=1.1pt, label=below:$w$] {};
			\node (x3) at (18,0)  [circle, fill=black, inner sep=1.1pt, label=below:$z$] {};
			\draw (x1) to[bend left=30] (x3);
			\draw [dashed, ->] (x2) to (x3);
            \draw (x2) to[bend right=30] (x0);
			\draw [dashed, ->] (x1) to (x0);
            
            \node (z0) at (21,0) [circle, fill=black, inner sep=1.1pt, label=below:$u$] {};
            \draw [->, dashed] (z0) edge [out=135,in=45,looseness=15] (z0);

		\end{tikzpicture}
	
	\caption{Forbidden patterns for balanced interval (catch) digraphs, for $u < v < w < z$, describing a necessary condition given in~\cite{sanchita-ghos}.}
	\label{fig:balanced}
\end{figure}

    In~\cite{SandipchronID}, Das, Francis, Hell, and Huang characterize chronological interval digraphs by means of a linear ordering of their vertices satisfying four properties (besides being reflexive): 
    $vu \not \in E - \mathcal{S}(E)$; $uw \in \mathcal{S}(E)$ implies $uv,vw \in \mathcal{S}(E)$; $uw \in E - \mathcal{S}(E)$ implies either $uv \in E - \mathcal{S}(E)$ or both $uv \in \mathcal{S}(E)$ and $vw \in E - \mathcal{S}(E)$; and $uw \not \in E$ implies either $uv \not \in E$ or $vw \not \in \mathcal{S}(E)$.
    These properties can be translated as the forbidden patterns in Figure~\ref{fig:chron}.

\begin{figure}[h]
\centering
\begin{tikzpicture}[font=\scriptsize]
\matrix[row sep=3mm, column sep=6mm]{

\node{\begin{tikzpicture}[scale=0.5, ->, >=stealth]
\node (z0) at (-4,0) [circle, fill=black, inner sep=1.1pt, label=below:$u$] {};
            \draw [->, dashed] (z0) edge [out=135,in=45,looseness=15] (z0);

\node (v0) at (-2,0) [circle, fill=black, inner sep=1.1pt, label=below:$u$] {};
\node (v1) at (0,0)  [circle, fill=black, inner sep=1.1pt, label=below:$v$] {};
\draw (v1) to[bend right=38] (v0);
\draw [dashed, ->] (v0) to (v1);
\end{tikzpicture}
}; &

\node{
\begin{tikzpicture}[scale=0.5, ->, >=stealth]
\node (v0) at (-2,0) [circle, fill=black, inner sep=1.1pt, label=below:$u$] {};
\node (v1) at (0,0)  [circle, fill=black, inner sep=1.1pt, label=below:$v$] {};
\node (v2) at (2,0)  [circle, fill=black, inner sep=1.1pt, label=below:$w$] {};
\draw (v0) to[bend left=38] (v2);
\draw [dashed, ->] (v0) to (v1);
\end{tikzpicture}
}; &

\node{
\begin{tikzpicture}[scale=0.5, ->, >=stealth]
\node (v0) at (-2,0) [circle, fill=black, inner sep=1.1pt, label=below:$u$] {};
\node (v1) at (0,0)  [circle, fill=black, inner sep=1.1pt, label=below:$v$] {};
\node (v2) at (2,0)  [circle, fill=black, inner sep=1.1pt, label=below:$w$] {};
\draw (v2) to[bend right=38] (v0);
\draw [dashed, ->] (v1) to (v0);
\end{tikzpicture}
}; & 

\node{
\begin{tikzpicture}[scale=0.5, ->, >=stealth]
\node (v0) at (-2,0) [circle, fill=black, inner sep=1.1pt, label=below:$u$] {};
\node (v1) at (0,0)  [circle, fill=black, inner sep=1.1pt, label=below:$v$] {};
\node (v2) at (2,0)  [circle, fill=black, inner sep=1.1pt, label=below:$w$] {};
\draw (v2) to[bend right=38] (v0);
\draw [dashed, ->] (v2) to (v1);
\end{tikzpicture}
}; \\

%%%%

\node{
\begin{tikzpicture}[scale=0.5, ->, >=stealth]
\node (v0) at (-2,0) [circle, fill=black, inner sep=1.1pt, label=below:$u$] {};
\node (v1) at (0,0)  [circle, fill=black, inner sep=1.1pt, label=below:$v$] {};
\node (v2) at (2,0)  [circle, fill=black, inner sep=1.1pt, label=below:$w$] {};
\draw (v2) to[bend right=38] (v0);
\draw [dashed, ->] (v1) to (v2);
\end{tikzpicture}
}; & 

\node{
\begin{tikzpicture}[scale=0.5, ->, >=stealth]
\node (v0) at (-2,0) [circle, fill=black, inner sep=1.1pt, label=below:$u$] {};
\node (v1) at (0,0)  [circle, fill=black, inner sep=1.1pt, label=below:$v$] {};
\node (v2) at (2,0)  [circle, fill=black, inner sep=1.1pt, label=below:$w$] {};
\draw (v0) to[bend left=38] (v2);
\draw [dashed, ->] (v1) to (v2);
\draw (v1) to (v0);
\end{tikzpicture}
}; &

%%%%

\node{
\begin{tikzpicture}[scale=0.5, ->, >=stealth]
\node (v0) at (-2,0) [circle, fill=black, inner sep=1.1pt, label=below:$u$] {};
\node (v1) at (0,0)  [circle, fill=black, inner sep=1.1pt, label=below:$v$] {};
\node (v2) at (2,0)  [circle, fill=black, inner sep=1.1pt, label=below:$w$] {};
\draw (v0) to[bend left=38] (v2);
\draw [dashed, ->] (v2) to[bend left=32] (v0);
\draw (v1) to (v0);
\draw (v2) to (v1);
\end{tikzpicture}
}; &

\node{
\begin{tikzpicture}[scale=0.5, ->, >=stealth]
\node (v0) at (-2,0) [circle, fill=black, inner sep=1.1pt, label=below:$u$] {};
\node (v1) at (0,0)  [circle, fill=black, inner sep=1.1pt, label=below:$v$] {};
\node (v2) at (2,0)  [circle, fill=black, inner sep=1.1pt, label=below:$w$] {};
\draw [dashed, ->] (v0) to[bend left=38] (v2);
\draw (v0) to (v1);
\draw (v1) to (v2);
\draw (v2) to (v1);
\end{tikzpicture}
}; \\
};
\end{tikzpicture}

\caption{Forbidden patterns for chronological interval digraphs, for $u < v < w$, extracted from the characterization in~\cite{SandipchronID}. }
\label{fig:chron}
\end{figure}
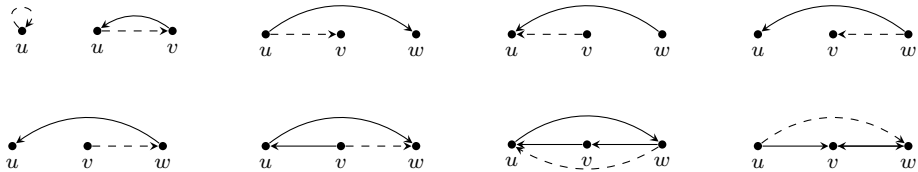
%\vspace{-3mm}
    
    Furthermore, in~\cite{Francis2021Kernel}, Francis, Hell, and Jacob prove that a digraph is a reflexive interval digraph if and only if its vertex set admits a linear ordering in which none of the forbidden structural patterns illustrated in Figure~\ref{fig:HFJ-patrones-reflexive} occur.
%    \vspace{-4mm}

\begin{figure}[h!]
\centering
\begin{tikzpicture}[font=\scriptsize]
\matrix[row sep=3mm, column sep=10mm]{

\node{\begin{tikzpicture}[scale=0.5, ->, >=stealth]
\node (v0) at (-2,0) [circle, fill=black, inner sep=1.1pt, label=below:$u$] {};
\node (v1) at (0,0)  [circle, fill=black, inner sep=1.1pt, label=below:$v$] {};
\node (v2) at (2,0)  [circle, fill=black, inner sep=1.1pt, label=below:$w$] {};
\node (v3) at (4,0)  [circle, fill=black, inner sep=1.1pt, label=below:$z$] {};
\draw (v0) to[bend left=38] (v3);
\draw [dashed, ->] (v0) to (v1);
\draw [dashed, ->] (v2) to (v3);
\node at (1.1,-1.4) {(a)};
\end{tikzpicture}
}; &

\node{
\begin{tikzpicture}[scale=0.5, ->, >=stealth]
\node (v0) at (-2,0) [circle, fill=black, inner sep=1.1pt, label=below:$u$] {};
\node (v1) at (0,0)  [circle, fill=black, inner sep=1.1pt, label=below:$v$] {};
\node (v2) at (2,0)  [circle, fill=black, inner sep=1.1pt, label=below:$w$] {};
\node (v3) at (4,0)  [circle, fill=black, inner sep=1.1pt, label=below:$z$] {};
\draw (v0) to[bend left=38] (v3);
\draw (v1) to  (v2);
\draw [dashed, ->] (v0) to[bend right=38](v2);
\draw [dashed, ->] (v1) to[bend right=38](v3);
\node at (1.1,-1.4) {(b)};
\end{tikzpicture}
}; & 

\node{
\begin{tikzpicture}[scale=0.5, ->, >=stealth]
\node (v0) at (-2,0) [circle, fill=black, inner sep=1.1pt, label=below:$u$] {};
\node (v1) at (0,0)  [circle, fill=black, inner sep=1.1pt, label=below:$v$] {};
\node (v2) at (2,0)  [circle, fill=black, inner sep=1.1pt, label=below:$w$] {};
\node (v3) at (4,0)  [circle, fill=black, inner sep=1.1pt, label=below:$z$] {};
\draw (v1) to[bend left=38] (v3);
\draw (v0) to[bend left=38] (v2);
\draw [dashed, ->] (v0) to[bend right=30] (v3);
\draw [dashed, ->] (v1) to[bend right=38] (v2);
\node at (1.1,-1.4) {(c)};
\end{tikzpicture}
}; \\

%%%%

\node{\begin{tikzpicture}[scale=0.5, ->, >=stealth]
\node (v0) at (-2,0) [circle, fill=black, inner sep=1.1pt, label=below:$u$] {};
\node (v1) at (0,0)  [circle, fill=black, inner sep=1.1pt, label=below:$v$] {};
\node (v2) at (2,0)  [circle, fill=black, inner sep=1.1pt, label=below:$w$] {};
\node (v3) at (4,0)  [circle, fill=black, inner sep=1.1pt, label=below:$z$] {};
\draw (v3) to[bend right=38] (v0);
\draw [dashed, ->] (v1) to (v0);
\draw [dashed, ->] (v3) to (v2);
\node at (1.1,-1.4) {(d)};
\end{tikzpicture}
}; &

\node{
\begin{tikzpicture}[scale=0.5, ->, >=stealth]
\node (v0) at (-2,0) [circle, fill=black, inner sep=1.1pt, label=below:$u$] {};
\node (v1) at (0,0)  [circle, fill=black, inner sep=1.1pt, label=below:$v$] {};
\node (v2) at (2,0)  [circle, fill=black, inner sep=1.1pt, label=below:$w$] {};
\node (v3) at (4,0)  [circle, fill=black, inner sep=1.1pt, label=below:$z$] {};
\draw (v3) to[bend right=38] (v0);
\draw (v2) to  (v1);
\draw [dashed, ->] (v2) to[bend left=38](v0);
\draw [dashed, ->] (v3) to[bend left=38](v1);
\node at (1.1,-1.4) {(e)};
\end{tikzpicture}
}; & 

\node{
\begin{tikzpicture}[scale=0.5, ->, >=stealth]
\node (v0) at (-2,0) [circle, fill=black, inner sep=1.1pt, label=below:$u$] {};
\node (v1) at (0,0)  [circle, fill=black, inner sep=1.1pt, label=below:$v$] {};
\node (v2) at (2,0)  [circle, fill=black, inner sep=1.1pt, label=below:$w$] {};
\node (v3) at (4,0)  [circle, fill=black, inner sep=1.1pt, label=below:$z$] {};
\draw (v3) to[bend right=38] (v1);
\draw (v2) to[bend right=38] (v0);
\draw [dashed, ->] (v3) to[bend left=30] (v0);
\draw [dashed, ->] (v2) to[bend left=38] (v1);
\node at (1.1,-1.4) {(f)};
\end{tikzpicture}
}; \\
};
\end{tikzpicture}

\caption{Forbidden patterns for reflexive interval digraphs, where possibly $v = w$ in (a), (b), (d) and (e) 
(figure reproduced from~\cite{Francis2021Kernel}). A dashed loop is the seventh forbidden pattern. }
\label{fig:HFJ-patrones-reflexive}
\end{figure}
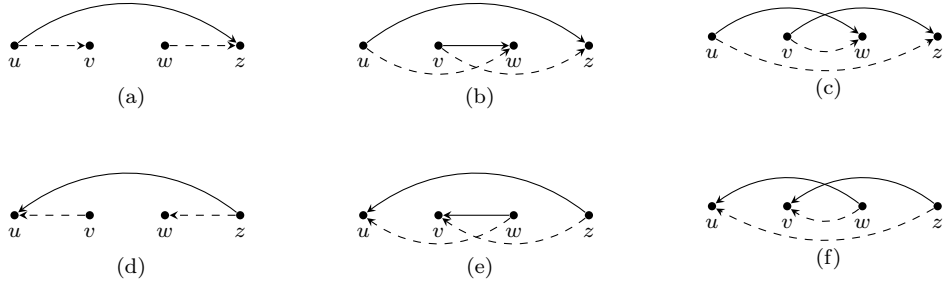
%\vspace{-3mm}

These results highlight a recurring theme: the structural constraints imposed by interval representations can often be captured by suitable vertex orderings that either avoid specific adjacency patterns or satisfy certain forbidden substructure properties. While classes such as adjusted, catch, point, chronological, and reflexive interval digraphs admit characterizations in terms of vertex orderings that forbid particular configurations---such as the patterns illustrated in Figures~\ref{fig:adj-catch} to \ref{fig:HFJ-patrones-reflexive}---no such vertex-ordering characterization has been established for the class of interval nest digraphs. This class is of significant interest due to its structural and algorithmic properties: Prisner demonstrated that interval nest digraphs are kernel-perfect and that their underlying graphs are weakly triangulated~\cite{Prisner1994}. Furthermore, he showed that if a nest representation is known, several problems that are generally NP-hard, including clique, chromatic number, independent set and kernel, can be solved in polynomial time. Although these digraphs were introduced and characterized by Prisner in terms of their interval representation, we have been unable to locate the characterization mentioned in~\cite{Sen1995}, and cited in its reference list as ``E. Prisner, \emph{Kernels in interval digraphs}, Inform. Proc. Lett., to appear''. To the best of our knowledge, no characterization of this class in terms of a pattern-avoiding vertex ordering currently appears in the literature.

In this work, we fill this gap by providing such a characterization in terms of a vertex ordering, which we call a \emph{nest ordering} and define through three structural properties. Specifically, we show that a finite digraph $D$ is an interval nest digraph if and only if it is reflexive and admits a nest ordering. Moreover, we prove that being reflexive and admitting a nest ordering is equivalent to admitting a linear vertex ordering that avoids the forbidden patterns shown in Figure~\ref{fig:fbd-patt}. Throughout this work, we focus on reflexive digraphs without multiple arcs; that is, every vertex has a loop and no ordered pair of vertices is joined by more than one arc.

\section{Characterization of interval nest digraphs}\label{sec:char}

In this section, we introduce a family of orderings of the vertices of a digraph and show that the existence of an ordering in this family characterizes interval nest digraphs.
Recall that $\mathcal{S}(E)$ is the set of all symmetric arcs.

\begin{defn}\label{Def:1}
Given a reflexive digraph \( D=(V,E) \), a total ordering $(\leq)$ of $V$ is called a \textit{nest ordering} if,  for all \( u \leq v \leq w \leq z \), all of the following hold:
\begin{enumerate}[label=(\roman*), ref=(\roman*), itemsep=1pt, topsep=2.5pt]
	\item \label{item:i}If \( uw, uz \in E \), then either \( uv \in E \) or \( vw,vz,wz \in \mathcal{S}(E) \). If \( zu, zv \in E \), then either \( zw \in E \) or \( wu,vu,wv \in \mathcal{S}(E) \).
	\item \label{item:ii} If \( uz, vw \in E \), then either \( uw \in E \) or \( vz \in E \). If \( zu, wv \in E \), then either \( zv \in E \) or \( wu \in E \).
	\item \label{item:iii}If \( uw, vz \in E \), then either \( vw \in E \) or \( uz \in E \). If \( zv, wu \in E \), then either \( zu \in E \) or \( wv \in E \).
\end{enumerate}
\end{defn} 
Notice that in the properties above, vertices in the sequence \( u \leq v \leq w \leq z \) are allowed to repeat. 

The first statements of properties~\ref{item:i}, \ref{item:ii}, and~~\ref{item:iii} are depicted in Figure~\ref{fig:fwd-patt}, while the full list of forbidden patterns can be found in Figure~\ref{fig:fbd-patt}.

%\vspace{-4mm}
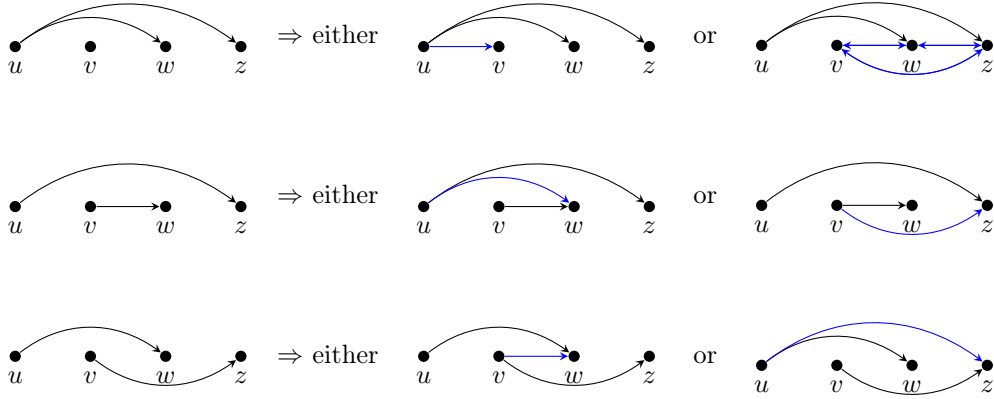
\begin{figure}[h]
\centering

\begin{tikzpicture}[font=\footnotesize]
\matrix[row sep=6mm, column sep=0mm]{

\node{\begin{tikzpicture}[scale=0.5, ->, >=stealth]
	\node (v0) at (-2,0) [circle, fill=black, inner sep=1.5pt, label=below:$u$] {};
	\node (v1) at (0,0) [circle, fill=black, inner sep=1.5pt, label=below:$v$] {};
	\node (v2) at (2,0) [circle, fill=black, inner sep=1.5pt, label=below:$w$] {};
	\node (v3) at (4,0) [circle, fill=black, inner sep=1.5pt, label=below:$z$] {};
	\draw (v0) to[bend left=38] (v3);
	\draw (v0)  to[bend left=38] (v2);
\end{tikzpicture}
}; & \node{$\Rightarrow$ either\phantom{j}}; &

\node{
\begin{tikzpicture}[scale=0.5, ->, >=stealth]
	\node (v0) at (-2,0) [circle, fill=black, inner sep=1.5pt, label=below:$u$] {};
	\node (v1) at (0,0) [circle, fill=black, inner sep=1.5pt, label=below:$v$] {};
	\node (v2) at (2,0) [circle, fill=black, inner sep=1.5pt, label=below:$w$] {};
	\node (v3) at (4,0) [circle, fill=black, inner sep=1.5pt, label=below:$z$] {};
	\draw (v0) to[bend left=38] (v3);
	\draw (v0)  to[bend left=38] (v2);
	\draw [color=blue] (v0) to (v1);
\end{tikzpicture}
}; & \node{\phantom{j}or\phantom{l}}; &

\node{
\begin{tikzpicture}[scale=0.5, ->, >=stealth]
	\node (v0) at (-2,0) [circle, fill=black, inner sep=1.5pt, label=below:$u$] {};
	\node (v1) at (0,0) [circle, fill=black, inner sep=1.5pt, label=below:$v$] {};
	\node (v2) at (2,0) [circle, fill=black, inner sep=1.5pt, label=below:$w$] {};
	\node (v3) at (4,0) [circle, fill=black, inner sep=1.5pt, label=below:$z$] {};
	\draw (v0) to[bend left=38] (v3);
	\draw (v0)  to[bend left=38] (v2);
	\draw [color=blue] (v1) to (v2);
    \draw [color=blue] (v2) to (v1);
    \draw [color=blue] (v2) to (v3);
    \draw [color=blue] (v3) to (v2);
    \draw [color=blue] (v3) to[bend left=38]  (v1);
    \draw [color=blue] (v1) to[bend right=38]  (v3);
\end{tikzpicture}
}; \\

%%%%

\node{
\begin{tikzpicture}[scale=0.5, ->, >=stealth]
	\node (v0) at (-2,0) [circle, fill=black, inner sep=1.5pt, label=below:$u$] {};
	\node (v1) at (0,0) [circle, fill=black, inner sep=1.5pt, label=below:$v$] {};
	\node (v2) at (2,0) [circle, fill=black, inner sep=1.5pt, label=below:$w$] {};
	\node (v3) at (4,0) [circle, fill=black, inner sep=1.5pt, label=below:$z$] {};
    \draw (v0) to[bend left=38]  (v3);
    \draw (v1) to(v2);
\end{tikzpicture}
}; & \node{$\Rightarrow$ either\phantom{j}}; &

\node{
\begin{tikzpicture}[scale=0.5, ->, >=stealth]
	\node (v0) at (-2,0) [circle, fill=black, inner sep=1.5pt, label=below:$u$] {};
	\node (v1) at (0,0) [circle, fill=black, inner sep=1.5pt, label=below:$v$] {};
	\node (v2) at (2,0) [circle, fill=black, inner sep=1.5pt, label=below:$w$] {};
	\node (v3) at (4,0) [circle, fill=black, inner sep=1.5pt, label=below:$z$] {};
    \draw (v0) to[bend left=38]  (v3);
    \draw (v1) to(v2);
	\draw [color=blue] (v0)  to[bend left=38] (v2);
\end{tikzpicture}
}; & \node{\phantom{j}or\phantom{l}}; &

\node{
\begin{tikzpicture}[scale=0.5, ->, >=stealth]
	\node (v0) at (-2,0) [circle, fill=black, inner sep=1.5pt, label=below:$u$] {};
	\node (v1) at (0,0) [circle, fill=black, inner sep=1.5pt, label=below:$v$] {};
	\node (v2) at (2,0) [circle, fill=black, inner sep=1.5pt, label=below:$w$] {};
	\node (v3) at (4,0) [circle, fill=black, inner sep=1.5pt, label=below:$z$] {};
    \draw (v0) to[bend left=38]  (v3);
    \draw (v1) to(v2);
  	\draw [color=blue] (v1)  to[bend right=38] (v3);
\end{tikzpicture}
}; \\

%%%%%

\node{
\begin{tikzpicture}[scale=0.5, ->, >=stealth]
	\node (v0) at (-2,0) [circle, fill=black, inner sep=1.5pt, label=below:$u$] {};
	\node (v1) at (0,0) [circle, fill=black, inner sep=1.5pt, label=below:$v$] {};
	\node (v2) at (2,0) [circle, fill=black, inner sep=1.5pt, label=below:$w$] {};
	\node (v3) at (4,0) [circle, fill=black, inner sep=1.5pt, label=below:$z$] {};
	\draw (v0) to[bend left=38] (v2);
	\draw (v1)  to[bend right=38] (v3);
\end{tikzpicture}
}; & \node{$\Rightarrow$ either\phantom{j}}; &

\node{
\begin{tikzpicture}[scale=0.5, ->, >=stealth]
	\node (v0) at (-2,0) [circle, fill=black, inner sep=1.5pt, label=below:$u$] {};
	\node (v1) at (0,0) [circle, fill=black, inner sep=1.5pt, label=below:$v$] {};
	\node (v2) at (2,0) [circle, fill=black, inner sep=1.5pt, label=below:$w$] {};
	\node (v3) at (4,0) [circle, fill=black, inner sep=1.5pt, label=below:$z$] {};
	\draw (v0) to[bend left=38] (v2);
	\draw (v1)  to[bend right=38] (v3);
    \draw [color=blue] (v1) to  (v2);
\end{tikzpicture}
}; 

& \node{\phantom{j}or\phantom{l}}; &

\node{
\begin{tikzpicture}[scale=0.5, ->, >=stealth]
	\node (v0) at (-2,0) [circle, fill=black, inner sep=1.5pt, label=below:$u$] {};
	\node (v1) at (0,0) [circle, fill=black, inner sep=1.5pt, label=below:$v$] {};
	\node (v2) at (2,0) [circle, fill=black, inner sep=1.5pt, label=below:$w$] {};
	\node (v3) at (4,0) [circle, fill=black, inner sep=1.5pt, label=below:$z$] {};
	\draw (v0) to[bend left=38] (v2);
	\draw (v1)  to[bend right=38] (v3);
    \draw [color=blue] (v0) to[bend left=38]  (v3);
\end{tikzpicture}
}; \\
};
\end{tikzpicture}
\caption{Forward patterns induced by a nest ordering, where $u \leq v \leq w \leq z$. Whenever the black arcs are present, at least one of the two column patterns must occur.}\label{fig:fwd-patt}
\end{figure}
%\vspace{-3mm}

The following theorem is the main result of this work.

\begin{thm}\label{characterization}
	Let \( D=(V,E) \) be a finite digraph. Then \( D \) is an interval nest digraph if and only if $D$ is reflexive and there exists a nest ordering on \( V \).	
\end{thm}
The proof of this theorem is rather involved and, for the sake of clarity, is divided into several lemmas. We begin by showing that every interval nest digraph admits a nest ordering.

\begin{lem}\label{lem:ida}
	Let \( D=(V,E) \) be a finite digraph. If \( D \) is an interval nest digraph then $D$ is reflexive and there exists a nest ordering on \( V \).	
\end{lem}
\begin{proof}
	
	Let \( D=(V,E) \) be an interval nest digraph, hence there exists an interval representation $\{I_v, J_v\}_{v\in V}$ of $D$ such that $J_v \subseteq I_v$ for each vertex \( v \in V \). In particular, $D$ is reflexive. 
	Without loss of generality, we assume all intervals have positive length. Indeed, any nest representation can be transformed to satisfy this by the following iterative procedure: moving from left to right along the real line, for every $x$ that is both a left and a right endpoint, we perturb the representation by increasing each right endpoint equal to $x$, as well as every endpoint strictly greater than $x$, by some $\varepsilon > 0$; it is not hard to see that this operation preserves the modeled interval digraph and the property of being a nest representation.
	In such a model, we can choose arbitrary points \( p_v \in J_v \) such that all \( \{p_v\}_{v \in V} \) are distinct. 
	This defines a total order \( \leq \) on the set \( V \) by setting $u \leq v$ if and only if $p_u \leq p_v$. 
    
    We will prove that an order defined this way satisfies the first statements in \ref{item:i}, \ref{item:ii}, and \ref{item:iii} of Definition \ref{Def:1}. As the second statements are symmetrical, the lemma then follows by applying this fact both to the representation $\{I_v, J_v\}_{v\in V}$ of $D$ with the order given by the points $\{p_v\}_{v\in V}$, and to the reflected representation with the order given by the reflected points (which is the inverse of the first); both represent the same digraph $D$.

    Given a bounded interval \( I \), we denote by \( \operatorname{left}(I) \) and \( \operatorname{right}(I) \) its left and right endpoints, respectively. We will use often that  for $u, v$ vertices of $V$, if $u < v$, then $uv \in E$ if and only if $\operatorname{left}(J_v) \leq \operatorname{right}(I_u)$, and $vu \in E$ if and only if $\operatorname{left}(I_v) \leq \operatorname{right}(J_u)$ (when $v < u$, the ``if'' part does not necessarily hold).

    We now verify that $\leq$ defined by the points \( \{p_v\}_{v \in V} \) 
    satisfies the first statement of~\ref{item:i}. 
    Let $u \leq v \leq w \leq z$, and suppose that \( uw, uz \in E \). If \( uv \in E \), the claim follows immediately. Thus, suppose \( uv \notin E \), which implies that  $v \neq w$,  \( I_u \cap J_v = \emptyset \), and, in particular, that \( p_v \in (\operatorname{right}(I_u), p_w) \). Since \( uw \in E \), we have that 
	$ \operatorname{left}(J_w) \leq \operatorname{right}(I_u) < p_v < p_w \leq \operatorname{right}(J_w),$
	from which it follows that \( p_v \in J_v \cap J_w \). Therefore \( I_v \cap J_w \neq \emptyset \) and \( I_w \cap J_v \neq \emptyset \), that is \( vw \in \mathcal{S}(E) \). Similarly, since \( uz \in E \), we have
	\begin{equation}\label{ib}
		\operatorname{left}(J_z) \leq \operatorname{right}(I_u) < p_v < p_w \leq p_z \leq \operatorname{right}(J_z),
	\end{equation}
	which implies that \( p_v \in J_v \cap J_z \), hence \( vz \in \mathcal{S}(E) \). Moreover, since \( p_w \in (p_v, p_z] \), it follows from \eqref{ib} that \( p_w \in J_w \cap J_z \), so \( wz \in \mathcal{S}(E)\). 

	We now show the first statement in~\ref{item:ii}. Let $u \leq v \leq w \leq z$, and suppose that \( uz, vw \in E \). If \( vz \in E \), the statement holds trivially. Suppose then that \( vz \not\in E \), which implies that \(\operatorname{right}(I_v) < \operatorname{left}(J_z).\)
	Since \( vw, uz \in E \), we have
	$\operatorname{left}(J_w)\leq \operatorname{right}(I_v)< \operatorname{left}(J_z) \leq\operatorname{right}(I_u),
	$
	which implies that \( \operatorname{left}(J_w) < \operatorname{right}(I_u) \). 
    Hence,  \( uw \in E \).
	
	Finally, let us verify the first statement of~\ref{item:iii}. Let $u \leq v \leq w \leq z$, and suppose that \( uw, vz \in E \). If \( vw \in E \), the claim follows; otherwise, assume \( vw \not\in E \).
  	Thus,
	\(
	\operatorname{right}(I_v) < \operatorname{left}(J_w),
	\)
	and since \( uw, vz \in E \), we have 
	$	\operatorname{left}(J_z) \leq \operatorname{right}(I_v) < \operatorname{left}(J_w) \leq \operatorname{right}(I_u).$
	This implies that \( \operatorname{left}(J_z) < \operatorname{right}(I_u) \), hence \( uz \in E \), completing the proof.
\end{proof}
In the remainder of this section, we will prove that if $D$ is reflexive and admits a nest ordering then $D$ is an interval nest digraph. To this end, we introduce some technical definitions that will be used in the proof.
\begin{defn}\label{def:stop}
	Given $D=(V,E)$ a reflexive digraph with a nest ordering $v_1, \dots, v_n$, we add two isolated auxiliary vertices  $v_{0}$ and $v_{n+1}$, such that  $v_0 < v_1 <  \dots < v_n < v_{n+1}$. Let $\tilde{V}= V \cup \{v_0, v_{n+1}\}$. 
    We define the \emph{stop} functions 
	$\stopmin, \stopmax:V\to \tilde{V}$ as: 
	\begin{align*}
	\stopmin(x)=\min_{z> x}\{z:z\notin N^+(x)\}, \quad  	\stopmax(x)=\max_{z< x}\{z:z\notin N^+(x)\}.  
	\end{align*}
	We define also $N^+_{\stopmisub}(x)=\{z>\stopmin(x):z\in N^+(x)\}$ and $N^+_{\stopmasub}(x)=\{z<\stopmax(x):z\in N^+(x)\}$.
\end{defn}
The inclusion of auxiliary isolated vertices ensures that both stop functions are well-defined. 
We now introduce the intervals used in our intersection model.
\begin{defn}\label{def:interval}
	Let $D=(V,E)$ be a reflexive digraph with a nest ordering. We define, for each $x \in V$, the intervals $I_x=[a_x,b_x]$, where
	$$
	a_x=\stopmax(x)+\frac{1}{2+|N^+_{\stopmasub}(x)|}\text{ and } b_x=\stopmin(x)-\frac{1}{2+|N^+_{\stopmisub}(x)|}
	$$
	and the intervals $J_x=[\alpha_x,\beta_x]$ where
	$$
	\alpha_x=\min\{\min_{z<x}\{b_z : x\in N^+(z)\},x\}\text{ and } \beta_x=\max\{\max_{z>x}\{a_z : x\in N^+(z)\},x\}.
	$$
	 
    Here, we adopt the convention of considering $\min(\emptyset) = +\infty$ and $\max(\emptyset) = -\infty$, thus if
    $\{b_z : x\in N^+(z)\}=\emptyset$ then $\alpha_x=x$, and if 
    $\{a_z : x\in N^+(z)\}=\emptyset$ then $\beta_x=x$.
\end{defn}
Intuitively, $I_x$ is defined to span from the last non-out-neighbor to the left to the first non-out-neighbor to the right; the fractional correction guarantees strict separation between neighboring endpoints and allows to realize distinct neighborhoods between vertices with the same stop function. We observe also that $x \in J_x$ for every $x \in V$. Next, we prove that the intervals defined above are indeed contained. % as shown in the following lemma.

From now on, we consider $D=(V,E)$ a reflexive digraph, unless otherwise stated.
\begin{lem}\label{lem:contention}
	Assume that $D$ admits a nest ordering. Then, for every $x\in V$, with $J_x$ and $I_x$ defined as in Definition \ref{def:interval}, we have $J_x\subseteq I_x$.
\end{lem}
\begin{proof}
We have to show that $a_x\leq \alpha_x$ and $\beta_x\leq b_x$. By symmetry, it suffices to prove $\beta_x\leq b_x$. Note that, if $\beta_x=x$ then, by definition, $x<\stopmin(x)$ and, since $b_x$ lies between $\stopmin(x)-\frac{1}{2}$ and $\stopmin(x)$, $\beta_x\leq b_x$ follows immediately. 
Thus, assume $\beta_x\neq x$. We must prove that 
$$\beta_x = \max_{z>x}\{a_z : x\in N^+(z)\} \leq \stopmin(x)-\frac{1}{2+|N^+_{\stopmisub}(x)|} = b_x.$$
Let $y$, with $y>x$, be the vertex that attains the maximum in $\beta_x$.  
That is,
\[
\beta_x=a_y=\stopmax(y)+\frac{1}{2+|N^+_{\stopmasub}(y)|}.
\]
We claim that $\stopmax(y) < \stopmin(x)$; assume to the contrary that $\stopmin(x)\leq \stopmax(y)$. Thus, we have $x < \stopmin(x)\leq \stopmax(y) < y$:

\begin{figure}[h]
	\centering
	\begin{tikzpicture}[scale=0.7, ->, >=stealth]
		\node (v0) at (-2,0) [circle, fill=black, inner sep=1.5pt, label=below:$x$] {};
		\node (v1) at (0,0) [circle, fill=black, inner sep=1.5pt, label=below:$\stopmin(x)$] {};
		\node (v2) at (2,0) [circle, fill=black, inner sep=1.5pt, label=below:$\stopmax(y)$] {};
		\node (v3) at (4,0) [circle, fill=black, inner sep=1.5pt, label=below:$y$] {};
		\draw (v3) to[bend right=30] (v0);
		\draw [dashed, ->] (v0) to (v1);
		\draw [dashed, ->] (v3) to (v2);
	\end{tikzpicture}
\end{figure}
Considering $x \leq x \leq \stopmax(y) < y$ (i.e., by repeating vertex $x$), since $yx \in E$, it follows from property~\ref{item:i} that either $y\stopmax(y)\in E$ or $\stopmax(y)x \in \mathcal{S}(E)$. Since $y\stopmax(y) \notin E$, then $\stopmax(y)x \in \mathcal{S}(E)$. 
Now, consider $x < \stopmin(x) \leq \stopmax(y) \leq\stopmax(y)$ (again repeating vertex $\stopmax(y)$). Using the fact that $x\stopmax(y) \in E$, property~\ref{item:i} implies that either $x\stopmin(x) \in E$ or $\stopmax(y)\stopmin(x) \in \mathcal{S}(E)$. Since $x\stopmin(x) \not\in E$, then $\stopmin(x)\stopmax(y) \in \mathcal{S}(E)$.
From this, we can easily deduce that $y\stopmin(x) \notin E$: if $y\stopmin(x) \in E$, then applying property~\ref{item:i} to the sequence $x < \stopmin(x) \leq \stopmax(y) < y$  we obtain that either $x\stopmin(x) \in \mathcal{S}(E)$ or $y\stopmax(y) \in E$, but neither statement holds. Therefore, $y\stopmin(y) \notin E$. Consequently, we apply property  \ref{item:i} to the sequence $x \leq x < \stopmin(x) < y$, and since $yx\in E$ we obtain that either $y\stopmin(x)\in E$ or $\stopmin(x)x\in \mathcal{S}(E)$, and none of them holds. Hence, $\stopmax(y) < \stopmin(x)$, and the claim follows.
Now, we have $\stopmin(x) - \stopmax(y) \geq 1$,  $\frac{1}{2+|N^+_{\stopmasub}(y)|}\leq \frac{1}{2}$, and $\frac{1}{2+|N^+_{\stopmisub}(x)|}\leq \frac{1}{2}$. Therefore,
\vspace{-0.3mm}
\begin{align*}
	b_x-\beta_x &= \stopmin(x)-\frac{1}{2+|N^+_{\stopmisub}(x)|} - \left(\stopmax(y)+\frac{1}{2+|N^+_{\stopmasub}(y)|}\right)\\
    &= \stopmin(x)-\frac{1}{2+|N^+_{\stopmisub}(x)|} - \stopmax(y)-\frac{1}{2+|N^+_{\stopmasub}(y)|}\\
	&\geq \stopmin(x)-\frac{1}{2} - \stopmax(y)-\frac{1}{2}\\
	&\geq 1-1\geq 0.
\end{align*}
Thus $b_x \geq \beta_x$, and we conclude the proof. \end{proof}

The following lemma establishes a structural property of the neighborhoods associated with the stop functions, which will be crucial when comparing endpoints of different intervals.
\begin{lem}\label{vecindades anidadas}
Assume that $D$ admits a nest ordering. Then, for any two vertices 
$x$ and $y$, the following hold: 
\begin{enumerate}
		\item If $\stopmin(x)=\stopmin(y)$ then either $N^+_{\stopmisub}(x)\subseteq N^+_{\stopmisub}(y)$ or $N^+_{\stopmisub}(y)\subseteq N^+_{\stopmisub}(x)$.
		\item If $\stopmax(x)=\stopmax(y)$ then either $N^+_{\stopmasub}(x)\subseteq N^+_{\stopmasub}(y)$ or $N^+_{\stopmasub}(y)\subseteq N^+_{\stopmasub}(x)$.
	\end{enumerate} 
\end{lem}

\begin{proof} 
We will prove only the first statement, the proof of the second is analogous. 
Let \(x,y\in V\) such that \(\stopmin(x)=\stopmin(y)\). The property is symmetric with respect to $x$ and $y$, and it holds trivially if $x=y$, so we will assume without loss of generality that \(x<y\). Let \(p\in V\) be such that \(\stopmin(x)=p=\stopmin(y)\). Then, by definition, \(x<y<p\).
	
Toward a contradiction, suppose that 
$N^+_{\stopmisub}(x) \not\subseteq N^+_{\stopmisub}(y)
	\text{ and } N^+_{\stopmisub}(y) \not\subseteq N^+_{\stopmisub}(x)$.
Then there exist distinct vertices \(u,w\in V\) such that
	\[
	u\in N^+_{\stopmisub}(x),\quad u\notin N^+_{\stopmisub}(y),
	\qquad
	w\in N^+_{\stopmisub}(y),\quad w\notin N^+_{\stopmisub}(x).
	\]
	By definition, \(p<u\) and \(p<w\). Therefore, there are two possible orderings: $x < y < p < u < w$ and $x < y < p < w < u$.
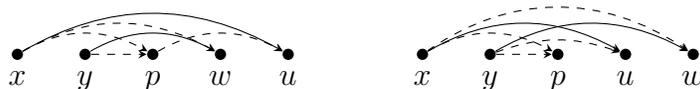
	\begin{figure}[h!]
		\centering
		%--- First diagram ---
		\begin{tikzpicture}[scale=0.45, ->, >=stealth]
			\node (u0) at (-2,0) [circle, fill=black, inner sep=1.5pt, label=below:$x$] {};
			\node (u1) at (0,0) [circle, fill=black, inner sep=1.5pt, label=below:$y$] {};
			\node (u2) at (2,0) [circle, fill=black, inner sep=1.5pt, label=below:$p$] {};
			\node (u3) at (4,0) [circle, fill=black, inner sep=1.5pt, label=below:$w$] {};
			\node (u4) at (6,0) [circle, fill=black, inner sep=1.5pt, label=below:$u$] {};
			\draw (u0)  to[bend left=30] (u4);
			\draw (u1)  to[bend left=30] (u3);
			\draw [dashed, ->] (u0) to [bend left=30] (u2);
			\draw [dashed, ->] (u1) to (u2);
			\draw [dashed, ->] (u0) to [bend left=30] (u3);
			\draw [dashed, ->] (u2) to [bend left=30] (u4);
		\end{tikzpicture}
        \hspace{1cm}
		%--- Second diagram ---
		\begin{tikzpicture}[scale=0.45, ->, >=stealth]
			\node (v0) at (-2,0) [circle, fill=black, inner sep=1.5pt, label=below:$x$] {};
			\node (v1) at (0,0) [circle, fill=black, inner sep=1.5pt, label=below:$y$] {};
			\node (v2) at (2,0) [circle, fill=black, inner sep=1.5pt, label=below:$p$] {};
			\node (v3) at (4,0) [circle, fill=black, inner sep=1.5pt, label=below:$u$] {};
			\node (v4) at (6,0) [circle, fill=black, inner sep=1.5pt, label=below:$w$] {};
			\draw (v0) to[bend left=30] (v3);
			\draw [dashed, ->] (v0) to [bend left=30] (v2);
			\draw (v1) to[bend left=30] (v4);
			\draw [dashed, ->] (v1) to (v2);
			\draw [dashed, ->] (v0) to [bend left=35] (v4);
			\draw [dashed, ->] (v1) to [bend left=20] (v3);
		\end{tikzpicture}
	\end{figure}

	We show that none of the orders $x<y<w<u$ and $x<y<u<w$ satisfies the nest properties. Indeed, since \(xu\in E\) and \(yw\in E\) but neither \(yu\in E\) nor \(xw\in E\) (by the assumptions on \(u\) and \(w\)), in the first case we are contradicting property~\ref{item:ii} and in the second case we are contradicting property~\ref{item:iii}. 
	Thus, the assumption leads to a contradiction, and we conclude that either $N^+_{\stopmisub}(x)\subseteq N^+_{\stopmisub}(y)$ or $N^+_{\stopmisub}(y)\subseteq N^+_{\stopmisub}(x)$. 
\end{proof}
This final lemma establishes the equivalence between adjacency in the digraph and the intersection of the previously defined intervals.
\begin{lem}\label{lem:arc_intersection}
Assume that $D$ admits a nest ordering. Then, for every $x, y \in V$, with $I_x, J_y$ defined as in Definition~\ref{def:interval}, $xy\in E$ if and only if $I_x\cap J_y\ne\emptyset$.
\end{lem}

\begin{proof} For $x=y$, the statement follows from Lemma~\ref{lem:contention}, since $D$ is reflexive and $J_y \subseteq I_x$ for every $x \in V$. Since all the definitions are symmetric, we will only prove the statement for $x<y$, since the case $x > y$ is analogous. 

Suppose first that $xy\in E$. Recall that, by definition,
$\alpha_y=\min\bigl\{\min_{z<y}\{b_z : y\in N^+(z)\},\,y\bigr\}$.
Since $xy\in E$, we have $b_x\in \{b_z : y\in N^+(z)\}$, and therefore $\alpha_y\le b_x$. Since $x < y$ and, by definition, $x \in I_x$ and $y \in J_y$, this implies that $I_x \cap J_y \neq \emptyset$.

Conversely, assume that $I_x \cap J_y \neq \emptyset$. In particular, \( \alpha_y \leq b_x \). Towards a contradiction, suppose that $xy \not \in E$, thus $\stopmin(x) \leq y$. Then we have $\alpha_y \leq b_x < \stopmin(x) \leq y$, and we conclude that $\alpha_y \neq y$, hence  $\alpha_y = \min_{z<y}\{b_z : y\in N^+(z)\}$. Let $m$ be the vertex such that $\alpha_y = b_m$. Then $m < y$, $my \in E$ (by definition of $m$ and $\alpha_y$), and $b_m \leq b_x$ (because $b_m = \alpha_y \leq b_x$). 
By definition, $b_m = \stopmin(m)-\frac{1}{2+|N^+_{\stopmisub}(m)|}$ and $b_x = \stopmin(x)-\frac{1}{2+|N^+_{\stopmisub}(x)|}$, and since $b_m \leq b_x$, it follows that $\stopmin(m) \leq \stopmin(x) \leq y$. Since $my \in E$, we can deduce that $\stopmin(m) < y$ and $y \in N^+_{\stopmisub}(m)$. 

If $\stopmin(m) = \stopmin(x)$, since $y \in N^+_{\stopmisub}(m)$ and $y \not \in N^+_{\stopmisub}(x)$, by Lemma~\ref{vecindades anidadas}, $N^+_{\stopmisub}(x) \subset N^+_{\stopmisub}(m)$ and $|N^+_{\stopmisub}(x)| < |N^+_{\stopmisub}(m)|$. But this implies that 
\begin{align*}
	b_x &= \stopmin(x)-\frac{1}{2+|N^+_{\stopmisub}(x)|}\\
    &< \stopmin(x)-\frac{1}{2+|N^+_{\stopmisub}(m)|} = \stopmin(m)-\frac{1}{2+|N^+_{\stopmisub}(m)|} = b_m,
\end{align*}
a contradiction to the fact that $b_m \leq b_x$.
Therefore, $m < \stopmin(m) < \stopmin(x) \leq y$. 

We have three cases according to the position of $x$: either 
$x < m < \stopmin(m) < \stopmin(x) \leq y$, or $m < x \leq \stopmin(m) < \stopmin(x) \leq y$, or $m < \stopmin(m) < x < \stopmin(x) \leq y$. In the first two cases, since $x \leq \stopmin(m) < \stopmin(x)$, we have $x\stopmin(m) \in E$. 
In these cases, since $x\stopmin(m) \in E$, $my \in E$, $m\stopmin(m) \not \in E$, and $xy \not \in E$, the sequences $x < m < \stopmin(m) < y$ and $m < x \leq \stopmin(m) < y$ contradict~\ref{item:iii} and~\ref{item:ii}, respectively. 

	\begin{figure}[h!]
		\centering
		%--- First diagram ---
		\begin{tikzpicture}[scale=0.5, ->, >=stealth, font=\footnotesize]
			\node (x) at (-2.7,0) [circle, fill=black, inner sep=1.5pt, label=below:\phantom{(}$x$\phantom{)}] {};
			\node (m) at (0,0) [circle, fill=black, inner sep=1.5pt, label=below:\phantom{(}$m$\phantom{)}] {};
			\node (sm) at (2.7,0) [circle, fill=black, inner sep=1.5pt, label=below:$\stopmin(m)$] {};
			\node (sx) at (5.4,0) [circle, fill=black, inner sep=1.5pt, label=below:$\stopmin(x)$] {};
			\node (y) at (8.1,0) [circle, fill=black, inner sep=1.5pt, label=below:\phantom{(}$y$\phantom{)}] {};
			\draw (x) to[bend left=30] (sm);
			\draw [dashed, ->] (x) to [bend left=30] (y);
			\draw [dashed, ->] (x) to [bend left=30] (sx);
			\draw (m) to[bend left=30] (y);
			\draw [dashed, ->] (m) to (sm);
		\end{tikzpicture}
		\hspace{1cm}
		%--- Second diagram ---
		\begin{tikzpicture}[scale=0.5, ->, >=stealth, font=\footnotesize]
			\node (m) at (-2.7,0) [circle, fill=black, inner sep=1.5pt, label=below:\phantom{(}$m$\phantom{)}] {};
			\node (x) at (0,0) [circle, fill=black, inner sep=1.5pt, label=below:\phantom{(}$x$\phantom{)}] {};
			\node (sm) at (2.7,0) [circle, fill=black, inner sep=1.5pt, label=below:$\stopmin(m)$] {};
			\node (sx) at (5.4,0) [circle, fill=black, inner sep=1.5pt, label=below:$\stopmin(x)$] {};
			\node (y) at (8.1,0) [circle, fill=black, inner sep=1.5pt, label=below:\phantom{(}$y$\phantom{)}] {};
			\draw (x) to (sm);
			\draw [dashed, ->] (x) to [bend left=30] (y);
			\draw [dashed, ->] (x) to [bend left=30] (sx);
			\draw (m) to[bend left=30] (y);
			\draw [dashed, ->] (m) to [bend left=30] (sm);
		\end{tikzpicture}
    \end{figure}

In the third case, either $mx \in E$ and the sequence $m < \stopmin(m) < x < y$ contradicts~\ref{item:i}, or $mx \not \in E$ and the sequence $m < x < y \leq y$ contradicts~\ref{item:i}. 

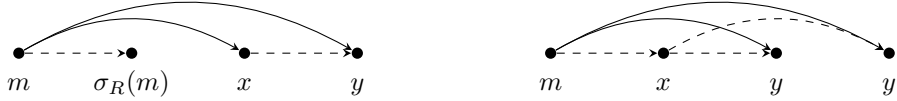
	\begin{figure}[h!]
		\centering
		%--- First diagram ---
		\begin{tikzpicture}[scale=0.5, ->, >=stealth, font=\footnotesize]
			\node (m) at (0,0) [circle, fill=black, inner sep=1.5pt, label=below:\phantom{(}$m$\phantom{)}] {};
			\node (sm) at (3,0) [circle, fill=black, inner sep=1.5pt, label=below:$\stopmin(m)$] {};
			\node (x) at (6,0) [circle, fill=black, inner sep=1.5pt, label=below:\phantom{(}$x$\phantom{)}] {};
			\node (y) at (9,0) [circle, fill=black, inner sep=1.5pt, label=below:\phantom{(}$y$\phantom{)}] {};
			\draw (m) to[bend left=30] (x);
			\draw [dashed, ->] (x) to (y);
			\draw (m) to[bend left=30] (y);
			\draw [dashed, ->] (m) to (sm);
		\end{tikzpicture}
		\hspace{1.5cm}
		%--- Second diagram ---
		\begin{tikzpicture}[scale=0.5, ->, >=stealth, font=\footnotesize]
			\node (m) at (0,0) [circle, fill=black, inner sep=1.5pt, label=below:\phantom{(}$m$\phantom{)}] {};
			\node (x) at (3,0) [circle, fill=black, inner sep=1.5pt, label=below:\phantom{(}$x$\phantom{)}] {};
			\node (y) at (6,0) [circle, fill=black, inner sep=1.5pt, label=below:\phantom{(}$y$\phantom{)}] {};
			\node (yy) at (9,0) [circle, fill=black, inner sep=1.5pt, label=below:\phantom{(}$y$\phantom{)}] {};
			\draw [dashed, ->] (m) to  (x);
			\draw [dashed, ->] (x) to (y);
			\draw [dashed, ->] (x) to [bend left=30] (yy);
			\draw (m) to[bend left=30] (y);
            \draw (m) to[bend left=30] (yy);
		\end{tikzpicture}
    \end{figure}

Since we have reached a contradiction in any case, we can conclude that  $xy \in E$, as desired. 
\end{proof}
\begin{proof}[Proof of Theorem~\ref{characterization}]
The preceding lemmas establish both directions of the correspondence between nest orderings and the interval representation of a nest digraph. First, we showed in Lemma~\ref{lem:ida} that every interval nest digraph admits a nest ordering. Then, given a reflexive digraph $D = (V,E)$ admitting a nest ordering, we define a family of intervals $\{I_x,J_x\}_{x\in V}$. We prove in Lemma~\ref{lem:contention} the containment \(J_x \subseteq I_x\) for every $x\in V$, and we prove in Lemma~\ref{lem:arc_intersection} (using Lemma~\ref{lem:contention} and Lemma~\ref{vecindades anidadas}) that, for $x,y \in V$, $xy \in E$ if and only if $I_x \cap J_y \neq \emptyset$. Together, these results allow us to conclude the full equivalence between interval nest digraphs and reflexive digraphs admitting a nest ordering.
\end{proof}

Our last result characterizes interval nest digraphs by forbidden patterns, showing that the difference with the forbidden patterns for reflexive interval digraphs (Figure~\ref{fig:HFJ-patrones-reflexive}) are just one pattern and its symmetrical one ((g) and~(h) in Figure~\ref{fig:fbd-patt}).

\begin{thm}\label{thm:patt-char}
A digraph is an interval nest digraph if and only if its vertex set admits a linear ordering in which none of the forbidden patterns illustrated in Figure~\ref{fig:fbd-patt} occur.    
\end{thm}

\begin{proof} By Theorem~\ref{characterization}, a digraph is an interval nest digraph if and only if it is reflexive and admits a nest ordering.  
Let us first suppose that $D = (V,E)$ is reflexive and admits a nest vertex ordering $\leq$. Since $D$ is reflexive, pattern~(j) cannot appear.
Patterns~(b) and~(e) violate, in a straightforward way, property~\ref{item:ii} in the definition of nest ordering. The same occurs for patterns~(c) and~(f) with property~\ref{item:iii}. For patterns~(a) and~(g), if $uw \in E$, then $u < v < w < z$ violate property~\ref{item:i}, and if $uw \not \in E$ then $u < w < z \leq z$ violate property~\ref{item:i} as well. For patterns~(d) and~(h), a similar argument can be applied with respect to the existence of the arc $zv$.  

Let us now suppose that $\leq$ is a linear vertex ordering of $D=(V,E)$ in which none of the forbidden structural patterns illustrated in Figure~\ref{fig:fbd-patt} occur. Since pattern~(j) does not appear, $D$ is reflexive. Let us verify that $\leq$ is also a nest ordering. 

Suppose first that the first statement of property~\ref{item:ii} does not hold. Then we have $u \leq v \leq w \leq z$ such that $uz, vw \in E$ and neither $uw \in E$ nor $vz \in E$. In particular, $u \neq v$ because $vw \in E$ and $uw \not \in E$, and $w \neq z$ since $vw \in E$ and $vz \not \in E$. Thus, pattern~(b) occurs, a contradiction. The proof for the second statement is symmetric, using  pattern~(e). 

Suppose now that the first statement of property~\ref{item:iii} does not hold. Then we have $u \leq v \leq w \leq z$ such that $uw, vz \in E$ and neither $vw \in E$ nor $uz \in E$. In particular, $u \neq v$ since $vw \not \in E$ and $uw \in E$, and $w \neq z$ because $vw \not \in E$ and $vz \in E$. Thus, pattern~(c) occurs, a contradiction. The proof for the second statement is symmetric, using  pattern~(f). 

Finally, suppose now that the first statement of property~\ref{item:i} does not hold. Then we have $u \leq v \leq w \leq z$ such that $uw, uz \in E$, $uv \not \in E$, and at least one of $vw$, $vz$, $wz$, $wv$, $zv$, $zw$ is not an arc. In particular, $u \neq v$ since $uv \not \in E$ and the pattern~(j) does not appear. If either $wz \not \in E$ or $zw \not \in E$ (in particular, $w \neq z$ because the pattern~(j) does not appear), then either pattern~(a) or pattern~(g) appears, leading to a contradiction. Thus $wz \in \mathcal{S}(E)$.  

If either $vz \not \in E$ or $zv \not \in E$ (in particular, $v \neq z$ because the pattern~(j) does not appear), then either pattern~(a) or pattern~(g) appears considering vertices $u < v = v < z$, leading to a contradiction. Thus $vz \in \mathcal{S}(E)$. 

Recall that $uw \in E$, so if either $vw \not \in E$ or $wv \not \in E$ (in particular, $v \neq w$ because the pattern~(j) does not appear), then either pattern~(a) or pattern~(g) appears considering vertices $u < v = v < w$, leading to a contradiction. Thus $vw \in \mathcal{S}(E)$ as well, and the statement holds. 

The proof for the second statement is symmetric, using  patterns~(d) and~(h). \end{proof}

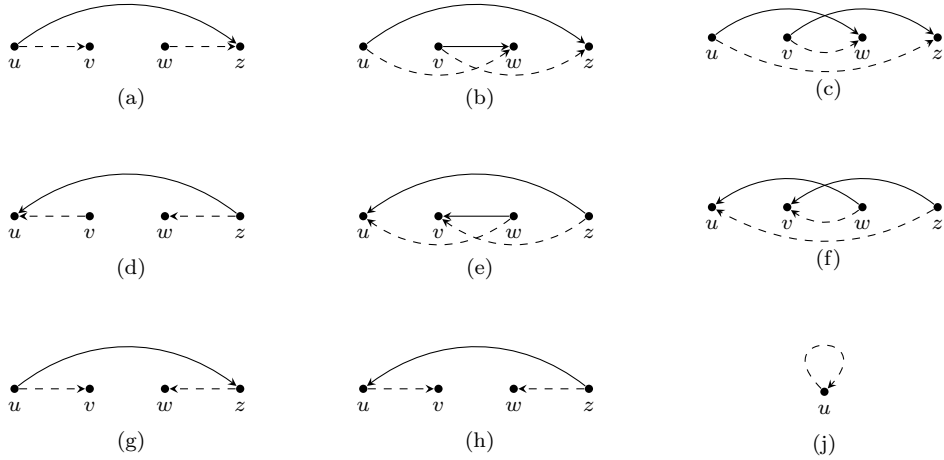
\begin{figure}[h]
\centering
\begin{tikzpicture}[font=\scriptsize]
\matrix[row sep=3mm, column sep=10mm]{

\node{\begin{tikzpicture}[scale=0.5, ->, >=stealth]
\node (v0) at (-2,0) [circle, fill=black, inner sep=1.1pt, label=below:$u$] {};
\node (v1) at (0,0)  [circle, fill=black, inner sep=1.1pt, label=below:$v$] {};
\node (v2) at (2,0)  [circle, fill=black, inner sep=1.1pt, label=below:$w$] {};
\node (v3) at (4,0)  [circle, fill=black, inner sep=1.1pt, label=below:$z$] {};
\draw (v0) to[bend left=38] (v3);
\draw [dashed, ->] (v0) to (v1);
\draw [dashed, ->] (v2) to (v3);
\node at (1.1,-1.4) {(a)};
\end{tikzpicture}
}; &

\node{
\begin{tikzpicture}[scale=0.5, ->, >=stealth]
\node (v0) at (-2,0) [circle, fill=black, inner sep=1.1pt, label=below:$u$] {};
\node (v1) at (0,0)  [circle, fill=black, inner sep=1.1pt, label=below:$v$] {};
\node (v2) at (2,0)  [circle, fill=black, inner sep=1.1pt, label=below:$w$] {};
\node (v3) at (4,0)  [circle, fill=black, inner sep=1.1pt, label=below:$z$] {};
\draw (v0) to[bend left=38] (v3);
\draw (v1) to  (v2);
\draw [dashed, ->] (v0) to[bend right=38](v2);
\draw [dashed, ->] (v1) to[bend right=38](v3);
\node at (1.1,-1.4) {(b)};
\end{tikzpicture}
}; & 

\node{
\begin{tikzpicture}[scale=0.5, ->, >=stealth]
\node (v0) at (-2,0) [circle, fill=black, inner sep=1.1pt, label=below:$u$] {};
\node (v1) at (0,0)  [circle, fill=black, inner sep=1.1pt, label=below:$v$] {};
\node (v2) at (2,0)  [circle, fill=black, inner sep=1.1pt, label=below:$w$] {};
\node (v3) at (4,0)  [circle, fill=black, inner sep=1.1pt, label=below:$z$] {};
\draw (v1) to[bend left=38] (v3);
\draw (v0) to[bend left=38] (v2);
\draw [dashed, ->] (v0) to[bend right=30] (v3);
\draw [dashed, ->] (v1) to[bend right=38] (v2);
\node at (1.1,-1.4) {(c)};
\end{tikzpicture}
}; \\

%%%%

\node{\begin{tikzpicture}[scale=0.5, ->, >=stealth]
\node (v0) at (-2,0) [circle, fill=black, inner sep=1.1pt, label=below:$u$] {};
\node (v1) at (0,0)  [circle, fill=black, inner sep=1.1pt, label=below:$v$] {};
\node (v2) at (2,0)  [circle, fill=black, inner sep=1.1pt, label=below:$w$] {};
\node (v3) at (4,0)  [circle, fill=black, inner sep=1.1pt, label=below:$z$] {};
\draw (v3) to[bend right=38] (v0);
\draw [dashed, ->] (v1) to (v0);
\draw [dashed, ->] (v3) to (v2);
\node at (1.1,-1.4) {(d)};
\end{tikzpicture}
}; &

\node{
\begin{tikzpicture}[scale=0.5, ->, >=stealth]
\node (v0) at (-2,0) [circle, fill=black, inner sep=1.1pt, label=below:$u$] {};
\node (v1) at (0,0)  [circle, fill=black, inner sep=1.1pt, label=below:$v$] {};
\node (v2) at (2,0)  [circle, fill=black, inner sep=1.1pt, label=below:$w$] {};
\node (v3) at (4,0)  [circle, fill=black, inner sep=1.1pt, label=below:$z$] {};
\draw (v3) to[bend right=38] (v0);
\draw (v2) to  (v1);
\draw [dashed, ->] (v2) to[bend left=38](v0);
\draw [dashed, ->] (v3) to[bend left=38](v1);
\node at (1.1,-1.4) {(e)};
\end{tikzpicture}
}; & 

\node{
\begin{tikzpicture}[scale=0.5, ->, >=stealth]
\node (v0) at (-2,0) [circle, fill=black, inner sep=1.1pt, label=below:$u$] {};
\node (v1) at (0,0)  [circle, fill=black, inner sep=1.1pt, label=below:$v$] {};
\node (v2) at (2,0)  [circle, fill=black, inner sep=1.1pt, label=below:$w$] {};
\node (v3) at (4,0)  [circle, fill=black, inner sep=1.1pt, label=below:$z$] {};
\draw (v3) to[bend right=38] (v1);
\draw (v2) to[bend right=38] (v0);
\draw [dashed, ->] (v3) to[bend left=30] (v0);
\draw [dashed, ->] (v2) to[bend left=38] (v1);
\node at (1.1,-1.4) {(f)};
\end{tikzpicture}
}; \\

\node{\begin{tikzpicture}[scale=0.5, ->, >=stealth]
\node (v0) at (-2,0) [circle, fill=black, inner sep=1.1pt, label=below:$u$] {};
\node (v1) at (0,0)  [circle, fill=black, inner sep=1.1pt, label=below:$v$] {};
\node (v2) at (2,0)  [circle, fill=black, inner sep=1.1pt, label=below:$w$] {};
\node (v3) at (4,0)  [circle, fill=black, inner sep=1.1pt, label=below:$z$] {};
\draw (v0) to[bend left=38] (v3);
\draw [dashed, ->] (v0) to (v1);
\draw [dashed, ->] (v3) to (v2);
\node at (1.1,-1.4) {(g)};
\end{tikzpicture}
}; &

\node{\begin{tikzpicture}[scale=0.5, ->, >=stealth]
\node (v0) at (-2,0) [circle, fill=black, inner sep=1.1pt, label=below:$u$] {};
\node (v1) at (0,0)  [circle, fill=black, inner sep=1.1pt, label=below:$v$] {};
\node (v2) at (2,0)  [circle, fill=black, inner sep=1.1pt, label=below:$w$] {};
\node (v3) at (4,0)  [circle, fill=black, inner sep=1.1pt, label=below:$z$] {};
\draw (v3) to[bend right=38] (v0);
\draw [dashed, ->] (v0) to (v1);
\draw [dashed, ->] (v3) to (v2);
\node at (1.1,-1.4) {(h)};
\end{tikzpicture}
}; & 

\node{\begin{tikzpicture}[scale=0.5, ->, >=stealth]
\node (v0) at (-2,0) [circle, fill=black, inner sep=1.1pt, label=below:$u$] {};
\draw [->, dashed,scale=3] (v0) edge [out=135,in=45,looseness=32] (v0);
\node at (-2,-1.4) {(j)};
\end{tikzpicture}
}; \\
};
\end{tikzpicture}

\caption{Forbidden patterns characterizing interval nest digraphs, where possibly $v = w$ in~(a), (b), (d), (e), (g) and~(h).}
\label{fig:fbd-patt}
\end{figure}

\section{Conclusions and Future Work}

The class of interval digraphs has been extensively studied, and several of its subclasses admit characterizations based on vertex orderings. 
The characterization of interval nest digraphs presented in this work fills a missing gap in this line of research and provides a new structural understanding of this subclass. 
In particular, it establishes the theoretical foundation upon which efficient recognition algorithms can be developed, analogous to those already known for interval digraphs and all their other mentioned subclasses~\cite{Muller97,Prisner1994,Takaoka21}. To the best of our knowledge and as reported in the cited literature, no polynomial time recognition algorithm is known for the class of interval nest digraphs.

Future research may explore the existence of efficient recognition algorithms based on nest orderings, the identification of minimal forbidden subdigraphs for this class, and possible connections between interval nest digraphs and other families of digraphs, thereby contributing to a broader understanding of the algorithmic and structural landscape of these classes.

\section*{Acknowledgments} 

This work was partially supported by  UBACyT 20020220300079BA and CONICET PIP 11220200100084CO and 11220250101365CO. We would like to thank Eric Brandwein for fruitful discussions about the presentation of the proof of the main theorem.

\end{document}